\documentclass[12pt,reqno]{amsart}

\usepackage{amsmath, amssymb, amsthm}

\begin{document}

 \bibliographystyle{plain}
 \newtheorem{theorem}{Theorem}
 \newtheorem{lemma}[theorem]{Lemma}
 \newtheorem{corollary}[theorem]{Corollary}
 \newtheorem{problem}[theorem]{Problem}
 \newtheorem{conjecture}[theorem]{Conjecture}
 \newtheorem{definition}[theorem]{Definition}
 \newtheorem{prop}[theorem]{Proposition}
 \numberwithin{equation}{section}
 \numberwithin{theorem}{section}

 \newcommand{\mo}{~\mathrm{mod}~}
 \newcommand{\mc}{\mathcal}
 \newcommand{\rar}{\rightarrow}
 \newcommand{\Rar}{\Rightarrow}
 \newcommand{\lar}{\leftarrow}
 \newcommand{\lrar}{\leftrightarrow}
 \newcommand{\Lrar}{\Leftrightarrow}
 \newcommand{\zpz}{\mathbb{Z}/p\mathbb{Z}}
 \newcommand{\mbb}{\mathbb}
 \newcommand{\B}{\mc{B}}
 \newcommand{\cc}{\mc{C}}
 \newcommand{\D}{\mc{D}}
 \newcommand{\E}{\mc{E}}
 \newcommand{\F}{\mathbb{F}}
 \newcommand{\G}{\mc{G}}
  \newcommand{\ZG}{\Z (G)}
 \newcommand{\FN}{\F_n}
 \newcommand{\I}{\mc{I}}
 \newcommand{\J}{\mc{J}}
 \newcommand{\M}{\mc{M}}
 \newcommand{\nn}{\mc{N}}
 \newcommand{\qq}{\mc{Q}}
 \newcommand{\PP}{\mc{P}}
 \newcommand{\U}{\mc{U}}
 \newcommand{\X}{\mc{X}}
 \newcommand{\Y}{\mc{Y}}
 \newcommand{\itQ}{\mc{Q}}
 \newcommand{\C}{\mathbb{C}}
 \newcommand{\R}{\mathbb{R}}
 \newcommand{\N}{\mathbb{N}}
 \newcommand{\Q}{\mathbb{Q}}
 \newcommand{\Z}{\mathbb{Z}}
 \newcommand{\A}{\mathbb{A}}
 \newcommand{\ff}{\mathfrak F}
 \newcommand{\fb}{f_{\beta}}
 \newcommand{\fg}{f_{\gamma}}
 \newcommand{\gb}{g_{\beta}}
 \newcommand{\vphi}{\varphi}
 \newcommand{\whXq}{\widehat{X}_q(0)}
 \newcommand{\Xnn}{g_{n,N}}
 \newcommand{\lf}{\left\lfloor}
 \newcommand{\rf}{\right\rfloor}
 \newcommand{\lQx}{L_Q(x)}
 \newcommand{\lQQ}{\frac{\lQx}{Q}}
 \newcommand{\rQx}{R_Q(x)}
 \newcommand{\rQQ}{\frac{\rQx}{Q}}
 \newcommand{\elQ}{\ell_Q(\alpha )}
 \newcommand{\oa}{\overline{a}}
 \newcommand{\oI}{\overline{I}}
 \newcommand{\dx}{\text{\rm d}x}
 \newcommand{\dy}{\text{\rm d}y}
\newcommand{\cal}[1]{\mathcal{#1}}
\newcommand{\cH}{{\cal H}}
\newcommand{\diam}{\operatorname{diam}}

\parskip=0.5ex

\title[Periodic points and primes in progressions]{Group automorphisms with prescribed growth of periodic points, and small primes in arithmetic progressions in intervals}
\author{Alan~Haynes and Christopher~White}
\subjclass[2010]{11N13, 37C35}
\thanks{Research of first author supported by EPSRC grants EP/J00149X/1 and EP/L001462/1. Research of second author supported by an EPSRC Doctoral Training Grant.}
\address{School of Mathematics, University of Bristol, Bristol, UK}
\email{alan.haynes@bris.ac.uk, chris.white@bris.ac.uk}

\allowdisplaybreaks

\begin{abstract}
We investigate the question of which growth rates are possible for the number of periodic points of a compact group automorphism. Our arguments involve a modification of Linnik's Theorem, concerning small prime numbers in arithmetic progressions which lie in intervals.
\end{abstract}

\maketitle

\section{Introduction}
In this paper we investigate the problem of determining which growth rates are possible for the number of periodic points of a compact group automorphism. We begin with a brief description of the history and motivation for pursuing this line of inquiry, which turns out to be closely connected to Lehmer's Conjecture concerning Mahler measures of algebraic numbers (an excellent survey of which is \cite{Smyt2008}).

Suppose $T$ is a (continuous) automorphism of a compact group $G$, let $h(T)$ denote the topological entropy of $T$, and for each $n\in\N$ let $F_n(T)$ be the number of points of period $n$,
\[F_n(T):=\#\{g\in G:T^ng=g\}.\]
We assume that our compact groups are Hausdorff, therefore the automorphisms $T$ are homeomorphisms. As mentioned in \cite{Lind1977}, Lehmer's Conjecture is equivalent to the statement that
\[\inf \{h(T):h(T)>0\}>0,\]
where the infimum is taken over all compact group automorphisms $T$. If $G$ is a metric space and $T$ is expansive then we also have (see \cite{LindSchmWard1990})
\begin{equation}\label{eqn.entropy<->perpnts}
h(T)=\lim_{n\rar\infty}\frac{\log F_n(T)}{n}.
\end{equation}
Motivated by this observation, Ward proved in \cite{Ward2005} that for any $C>0$ there is a compact group automorphism $T$ for which the limit on the right hand side of (\ref{eqn.entropy<->perpnts}) exists and is equal to $C$. However this does not imply Lehmer's Conjecture because, although the groups $G$ in Ward's construction are metrizable, the corresponding automorphisms $T$ are not expansive and have $h(T)=0$.

Although it is a further departure from Lehmer's Conjecture, the problem of constructing compact group automorphisms with prescribed growth rates of periodic points is a natural one which is interesting in its own right. In this paper we will demonstrate how a slightly modified version of Ward's basic method, which we call the $\F$-method, can be used to prove the following theorem.
\begin{theorem}\label{thm.perpnts1}
Suppose that $r:\N\rar\R$ is a multiplicative function satisfying
\[r(p^{a})-r(p^{a-1})>20.1\cdot a\log p,\]
for all primes $p$ and for all $a\in\N$. Then there exists a compact group automorphism $T$ with
\begin{equation*}
\frac{F_{n}(T)}{\exp(r(n))}\asymp 1.
\end{equation*}
\end{theorem}
Here and throughout the paper the notation $f(n)\asymp g(n)$ means that there exist constants $c_1,c_2>0$ such that
\[c_1|g(n)|<|f(n)|<c_2|g(n)|~\text{for all}~n\in\N.\]
With a small change to account for finitely many prime powers, Theorem \ref{thm.perpnts1} can be applied with $r(n)=Cn$ to recover Ward's result concerning the logarithmic growth rate of the number of periodic points. However if we are only interested in the logarithmic growth rate then, by arguing directly in the proof of the theorem, we can make the following improvement.
\begin{theorem}\label{thm.perpnts2}
Suppose that $r:\N\rar\R$ is a multiplicative function satisfying
\[r(p^{a})-r(p^{a-1})>13.4\cdot a\log p,\]
for all primes $p$ and for all $a\in\N$. Then there exists a compact group automorphism $T$ with
\begin{equation*}
\lim_{n\rar\infty}\frac{\log F_{n}(T)}{r(n)}= 1.
\end{equation*}
\end{theorem}
These two theorems follow from the more general Theorem \ref{thm.perpnts4} below, where the assumption that $r$ be multiplicative is relaxed. However, within the context of group automorphisms constructed using the $\F$-method, some non-trivial arithmetical conditions on $r$ must be enforced. We illustrate this fact in our proof of the following theorem, which demonstrates that even very quickly growing functions $r$ can fail to be the logarithmic growth rates for the number of periodic points of any group automorphism constructed using the $\F$-method.
\begin{theorem}\label{thm.perpnts3.1}
Given any function $t:\N\rar\R$, we can find a function $r:\N\rar [0,\infty)$ with $t(n)=o(r(n))$, such that no automorphism $T$ constructed with the $\F$-method satisfies
\[\lim_{n\rar\infty} \frac{\log F_n(T)}{r(n)}=1.\]
\end{theorem}
To resolve any potential confusion, we mention that this theorem contradicts a remark at the end of \cite{Ward2005}. The source of this contradiction is a small error in \cite[Equation (9)]{Ward2005}. The error is easily repaired and does not change the statement of the main result in that paper, but it does have some bearing on the first remark at the end of the paper.

Theorem \ref{thm.perpnts3.1} shows that there are logarithmic growth rates $r(n)$ which tend to infinity arbitrarily quickly, which cannot be obtained using the $\F$-method. It is also true that `polynomial' growth rates, of the form $r(n)=k\log n$, for $k\in\N$, cannot be obtained (this will be proven in Section \ref{sec.conclusion}). However it turns out that some growth rates which are only marginally faster than polynomial are attainable. To illustrate this, we will prove the following theorem.
\begin{theorem}\label{thm.perpnts3.2}
For $x>0$ and $k\in\N$, let ${^kx}$ denote the $k$th iterated exponential of $x$,
\[{^kx} := \underbrace{x^{x^{\cdot^{\cdot^{x}}}}}_k~,\]
and define the function $\iota:\N\rar\R$ by
\[\iota (n):=\min\{k\in\N:n<{^ke}\}.\]
There exists a compact group automorphism $T$ with
\begin{equation*}
\lim_{n\rar\infty}\frac{\log F_{n}(T)}{\iota(n)\log n}=1.
\end{equation*}
\end{theorem}
The $\F$-method (defined precisely in the next section) requires us to select a sequence of prime numbers $\{p_n\}_{n=1}^{\infty}$, with each $p_n=1\mo n$. In order to control the growth of the quantities $F_n(T)$ we choose the primes $p_n$ to lie in intervals whose position and length depend upon $n$ (and upon our desired value for $r(n)$). Ward was able to complete his proof (which uses a slightly different setup then ours) by appealing to Linnik's Theorem \cite{Linn1944} on primes in arithmetic progressions. Linnik's Theorem implies that there is a constant $\kappa>0$ such that, for all sufficiently large $n$, and for all $a\in\N$ with $(a,n)=1$, the smallest prime which is $a$ modulo $n$ is less than $n^\kappa$. After many subsequent improvements, Heath-Brown showed in \cite{Heat1992} that $\kappa$ can be taken to be $5.5$, and this has recently been improved by Xylouris in \cite{Xylo2009} to $5.2$. For our results we will prove the following extensions of Linnik's Theorem, which give information about small primes in arithmetic progressions which lie in intervals of prescribed lengths.
\begin{theorem}\label{thm.primes1}
If $\kappa>13.4$ then for any $\epsilon>0$, for any $a,n\in\N$ with $n$ sufficiently large (depending on $\epsilon$) and $(a,n)=1$, and for any $x>n^\kappa$, there is a prime in the interval
\[[x,(1+\epsilon)x)\]
which is equal to $a$ modulo $n$.
\end{theorem}
Theorem \ref{thm.primes1} is a version of Bertrand's Postulate for primes in arithmetic progressions. The next theorem is similar, but with smaller intervals.
\begin{theorem}\label{thm.primes2}
If $\kappa>20.1$ then there is an $\epsilon\ge 0$ such that, for any $a,n\in\N$ with $n$ sufficiently large and $(a,n)=1$, and for any $x>n^\kappa$, there is a prime in the interval
\[\left[x,x+\frac{x}{n^{1+\epsilon}}\right)\]
which is equal to $a$ modulo $n$.
\end{theorem}
We remark that similar problems were considered in \cite{Foge1965} and \cite{Welc2007}, however our results are stronger. For example, in \cite[Theorem 6.3.1]{Welc2007} (see also \cite[Theorem 4.1]{Welc2008}) it is proven that for $n$ sufficiently large, if $(a,n)=1$ and if $\kappa=328$ and $\theta=655/656$, then there is a prime number in the interval
\[\left[n^\kappa,n^\kappa+n^{\kappa\theta}\right)\]
which is $a\mo n$. It is not difficult to deduce from our Theorem \ref{thm.primes2} that this can be improved to $\kappa=20.1$ and $\theta=24/25$, and in fact, by our proof, better constants than this could also be obtained.

If we assume the Generalized Riemann Hypothesis then it is relatively easy to show that the constant $\kappa$ in Theorem \ref{thm.primes1} could be taken to be any real number larger than $2$, while the constant $\kappa$ in Theorem \ref{thm.primes2} could be taken to be any real number larger than $4$. However we do not make this assumption.

The layout for this paper is as follows. In Section \ref{sec.prelims} we define the $\F$-method and prove Theorems \ref{thm.perpnts1} and \ref{thm.perpnts2}, assuming the truth of Theorems \ref{thm.primes1} and \ref{thm.primes2}. In Section \ref{sec.primesthms} we prove the latter two theorems, and in Section \ref{sec.conclusion} we prove Theorems \ref{thm.perpnts3.1} and \ref{thm.perpnts3.2}.

Our notation and conventions are: $\varphi$ denotes the Euler phi function, $\mu$ the M\"{o}bius function, $\Lambda$ the von Mangoldt function. We use $d(n)$ to denote the number of positive divisors of an integer $n$, $\omega (n)$ to denote the number of distinct prime divisors of $n$, and $(m,n)$ to denote greatest common divisor of $m,n\in\Z$.  All summations are assumed to be restricted to positive integers. If $p$ is prime and $a\in\N$ then the notation $p^a\|n$ means that $p^a|n$ but $p^{a+1}\nmid n$. If $f$ and $g$ are complex valued functions with domain $X$ then the notation $f(x)\ll g(x)$, or $f(x)=O(g(x))$, means that there exists a constant $c>0$ such that $|f(x)|\le c|g(x)|$ for all $x\in X$. If $X=\N,\Z$, or $\R$ then the notation $f(x)=o(g(x))$ means that $f(x)/g(x)\rar 0$ as $x\rar\infty$. We say that a function $f:\N\rar\C$ is multiplicative if $f(mn)=f(m)f(n)$, for all $m,n\in\N$ with $(m,n)=1$.

\section{The $\F$-method and proofs of Theorems \ref{thm.perpnts1} and \ref{thm.perpnts2}}\label{sec.prelims}
Following Ward's construction, we consider sequences of pairs of integers $\{(p_n, g_n)\}_{n=1}^{\infty}$ satisfying the following properties:
\begin{enumerate}
 \item[(i)] Each $p_n$ is either prime or equal to $1$, and $p_1=1$,
 \item[(ii)] Each $p_n$ satisfies $p_n=1\mo n$,
 \item[(iii)] If $p_n=1$ then $g_n=0$, otherwise $g_n$ is a primitive root $\mo p_n$.
\end{enumerate}
Formally, we define the $\F$-method to be a function from the collection of all such sequences to the collection of pairs $(G,T)$ where $G$ is a compact group and $T$ is an automorphism of $G$. The group $G$ associated to the sequence $\{(p_n,g_n)\}$ is
\[G=\prod_{n=1}^{\infty}\F_{p_n},\]
the direct product (with the product topology) of the additive groups of the finite fields $\F_{p_n}$, each taken with the discrete topology. When $p_n=1$ we adopt the convention that $\F_{p_n}$ is the group with one element. The automorphism $T$ is defined by
\[T(x_1,x_2,\ldots,x_n,\ldots)=(T_1(x_1),T_2(x_2),\ldots,T_n(x_n),\ldots),\]
where $T_n:\F_{p_n}\rar\F_{p_n}$ is the trivial automorphism when $p_n=1,$ and is otherwise defined by
\[T_n(x_n)=g_n^{(p_n-1)/n}x_n.\]
If a group-automorphism pair $(G,T)$ lies in the image of the $\F$-method then we say that it can be constructed (or simply that $T$ can be constructed) using the $\F$-method.

Now assume that $(G,T)$ is the image under the $\F$-method of the sequence $\{(p_n,g_n)\}$. For any $m\in\N$, if $x_m\in\F_{p_m}$ is not the additive identity then the least period of $x_m$ under the map $T_m$ is equal to $m$. It follows that, for each $n\in\N$,
\begin{align}
F_n(T)&=\#\left\{\mathbf{x}\in G: \mathrm{lcm}\{m:p_m\not=1,~x_m\in\F_{p_m}^*\}|n\right\}=\prod_{d|n}p_d.\label{eqn.F_n formula}
\end{align}
Here we remark that the reason for allowing some of the integers $p_n$ to be $1$ is to give us the flexibility to slow the growth of $F_n(T)$ along certain subsequences of $n$. We will return to this thought in the final section. Now we use equation (\ref{eqn.F_n formula}) to prove the following slightly more general version of Theorems \ref{thm.perpnts1} and \ref{thm.perpnts2}, assuming for the moment the validity of Theorems \ref{thm.primes1} and \ref{thm.primes2}.
\begin{theorem}\label{thm.perpnts4}
Given $r:\N\rar\R$, let $s:\N\rar\R$ be defined by
\begin{equation}\label{eqn.s defn}
s(n):=\sum_{d|n}\mu (d)r(n/d).
\end{equation}
Then:
\begin{enumerate}
\item[(i)] If for all sufficiently large $n$ with $s(n)\not=0$ we have that
\[s(n)>20.1\cdot \log n,\]
then there exists a compact group automorphism $T$ with
\begin{equation*}
\frac{F_{n}(T)}{\exp(r(n))}\asymp 1.
\end{equation*}
\item[(ii)] If for all sufficiently large $n$ we have that
\[s(n)>13.4\cdot \log n,\]
then there exists a compact group automorphism $T$ with
\begin{equation*}
\lim_{n\rar\infty}\frac{\log F_{n}(T)}{r(n)}=1.
\end{equation*}
\end{enumerate}
\end{theorem}
\begin{proof}
To prove (i), choose $\epsilon>0$. By Theorem \ref{thm.primes2}, for every sufficiently large $n$ for which $s(n)\not= 0$, we can find a prime number $p_n=1\mo n$ which lies in the interval
\[\left[\exp(s(n)),\exp(s(n))\left(1+\frac{1}{n^{1+\epsilon}}\right)\right).\]
Then using (\ref{eqn.F_n formula}) we have that
\begin{align}
\log F_n(T)&=\sum_{d|n}s(d)+\sum_{d|n}\log\left(1+O\left(\frac{1}{d^{1+\epsilon}}\right)\right)\nonumber\\
&=r(n)+\sum_{d|n}\log\left(1+O\left(\frac{1}{d^{1+\epsilon}}\right)\right),\label{eqn.primesprf1}
\end{align}
where the last equality follows by M\"{o}bius inversion in (\ref{eqn.s defn}). For the error term we have that
\begin{align*}
\sum_{d|n}\log\left(1+O\left(\frac{1}{d^{1+\epsilon}}\right)\right)&\ll\sum_{d|n}\frac{1}{d^{1+\epsilon}}\ll 1,
\end{align*}
so exponentiating both sides of (\ref{eqn.primesprf1}) finishes the proof.

The proof of (ii) is almost the same, except that we appeal to Theorem \ref{thm.primes1} to choose, for all sufficiently large $n$, a prime $p_n=1\mo n$ in the interval
\[\left[\exp(s(n)),2\exp (s(n))\right).\]
Then for each $n$ we have that
\begin{align*}
\log F_n(T)&=r(n)+\sum_{d|n}\log (1+\xi_d),
\end{align*}
where each $\xi_d$ is a real number in the interval $[0,1)$. The error term here is less than $(\log 2)d(n)$, and since
\[r(n)=\sum_{d|n}s(d)\gg\sum_{d|n}\log d=\frac{d(n)\log n}{2},\]
we have that
\[\frac{\log F_n(T)}{r(n)}=1+O\left(\frac{1}{\log n}\right).\]
\end{proof}
Finally, to derive Theorems \ref{thm.perpnts1} and \ref{thm.perpnts2} from Theorem \ref{thm.perpnts4}, suppose that $\kappa>4$ and that $r(n)$ is multiplicative function which satisfies
\begin{equation}\label{eqn.primesprf2}
r(p^{a})-r(p^{a-1})>\kappa a\log p,
\end{equation}
for all primes $p$ and for all $a\in\N$. Then we have that
\[s(n)=\prod_{p^a\|n}(r(p^{a})-r(p^{a-1}))>\prod_{p^a\|n}\kappa a\log p.\]
Since $\kappa a\log p>2$ for all $p$ and $a$, the product on the right hand side is greater than
\[\sum_{p^a\|n}\kappa a\log p=\kappa\log n.\]
Substituting $\kappa=20.1$ and $\kappa=13.4$ and appealing to Theorem \ref{thm.perpnts4} finishes the proofs.

\section{Proofs of Theorems \ref{thm.primes1} and \ref{thm.primes2}}\label{sec.primesthms}
In this section we assume familiarity with the basic tools of analytic number theory, as presented for example in \cite{Dave1967}. Our proofs of Theorems \ref{thm.primes1} and \ref{thm.primes2} are based on deep but well known results which describe the distribution of zeros of Dirichlet L-functions in the critical strip. Our approach is modeled after the proofs of \cite[Theorem 7]{Gall1970} and \cite[Theorem 2]{Gall1972}.

To begin, let $n,a\in\N$ satisfy $(a,n)=1$ and for $x>0$ let
\[\psi (x,a,n)=\sum_{\substack{m\le x\\m=a~\mathrm{mod}~n}}\Lambda (m).\]
If $\chi$ is a Dirichlet character to modulus $n$ then define
\[\psi(x,\chi)=\sum_{m\le x}\chi (m)\Lambda (m),\]
and note by the orthogonality relations that
\begin{equation}\label{eqn.orth rels}
\psi(x,a,n)=\frac{1}{\varphi (n)}\sum_{\chi~\mathrm{mod}~n}\overline{\chi} (a)\psi(x,\chi).
\end{equation}
The approximate explicit formula for $\psi (x,\chi)$ (see \cite[Section 19]{Dave1967}) states that, for $2\le T\le x$,
\[\psi (x,\chi)=\delta_{\chi_0}x-\sum_{|\gamma|\le T}\frac{x^\rho}{\rho}+O\left(\frac{x(\log nx)^2}{T}+x^{1/4}\log x\right),\]
where $\delta_{\chi_0}$ is $1$ if $\chi$ is the principal character and $0$ otherwise. The sum on the right hand side here is a sum over all non-trivial zeros $\rho=\beta+i\gamma$ of the Dirichlet L-function $L(s,\chi)$, which have $|\gamma|\le T$. Substituting this in (\ref{eqn.orth rels}) we have for $0<h\le x$ that
\begin{align}
&\psi(x+h,a,n)-\psi(x,a,n)\nonumber\\
&\quad=\frac{h}{\varphi (n)}-\frac{1}{\varphi (n)}\sum_{\chi~\mathrm{mod}~n}\overline{\chi} (a)\sum_{|\gamma|\le T}\frac{(x+h)^\rho-x^\rho}{\rho}\nonumber\\
&\quad+O\left(\frac{x(\log nx)^2}{T}+x^{1/4}\log x\right).\label{eqn.psi diff err term 1}
\end{align}
For the double sum here we have the bound
\begin{align*}
\left|\sum_{\chi~\mathrm{mod}~n}\overline{\chi} (a)\sum_{|\gamma|\le T}\frac{(x+h)^\rho-x^\rho}{\rho}\right|&\le \sum_{\chi~\mathrm{mod}~n}\sum_{|\gamma|\le T}\left|\int_x^{x+h}y^{\rho-1}dy\right|\\
&\le h\sum_{\chi~\mathrm{mod}~n}\sum_{|\gamma|\le T}x^{\beta-1}.
\end{align*}
The main goal of the rest of the proof will be to show that for our choices of $x,h,$ we can choose $T$, depending on $n$, so that
\[\sum_{\chi~\mathrm{mod}~n}\sum_{|\gamma|\le T}x^{\beta-1}<1.\]
The other error terms involved will all be asymptotically negligible as $n\rar\infty$, which will imply that
\[\psi(x+h,a,n)-\psi (x,a,n)>0,\]
for all large enough $n$, and the conclusions of the theorems will easily follow. The proof is split into two cases, depending on whether or not there is a Siegel zero corresponding to a Dirichlet character of modulus $n$.\vspace*{.1in}

\noindent {\it Case 1:} Suppose that $n$ is sufficiently large and that $1-\beta_0/\log n$ is a Siegel zero to modulus $n$. We know from Siegel's theorem that, for any $\epsilon_1>0$,
\begin{equation}\label{eqn.sieg thm}
\beta_0\gg_{\epsilon_1} n^{-\epsilon_1}.
\end{equation}
All other zeros of Dirichlet L-functions to modulus $n$ must satisfy
\[\beta<1-\frac{c_2(\log\beta_0^{-1})}{\log (n(2+|\gamma|))},\]
and the constant $c_2$ may be taken to be any real number less than $2/3$ (see \cite{Grah1977} or \cite[p.267]{Heat1992}). Therefore, writing
\begin{equation}\label{eqn.eta def, sieg zero case}
\eta =\frac{c_2(\log\beta_0^{-1})}{\log (n(2+T))},
\end{equation}
and denoting the sum over non-Siegel zeros with a prime, we have that
\begin{align}
\sum_{\chi~\mathrm{mod}~n}\sideset{}{'}\sum_{|\gamma|\le T}x^{\beta-1}&=-\int_0^{1-\eta}x^{\alpha-1}d_\alpha N_n(\alpha,T)\nonumber\\
&=\frac{N_n(0,T)}{x}+\int_0^{1-\eta}x^{\alpha-1}N_n(\alpha,T)\log x~d\alpha,\label{eqn.non sieg zero int 1}
\end{align}
where $N_n(\alpha,T)$ denotes the number of zeros of all L-functions of modulus $n$ with $\alpha<\beta<1$ and $|\gamma|\le T$.
A zero-density estimate of Huxley \cite{Huxl1974} implies that, for any $\epsilon_2>0$, there exists an integer $n_0=n_0(\epsilon_2)$ such that, for all $T\ge 2$,
\begin{equation*}
N_n(\alpha, T)\le (nT)^{(12/5+\epsilon_2)(1-\alpha)}~\text{for all}~n\ge n_0.
\end{equation*}
For $\alpha\le 7/12$ it is better to use the standard 
estimate that, assuming $n_0$ has been chosen large enough,
\begin{equation*}
N_n(\alpha, T)\le (nT)^{1+\epsilon_2}~\text{for all}~n\ge n_0.
\end{equation*}
For the rest of the proof we will assume that $\epsilon_2$ is a fixed small number, to be specified later, and that $n\ge n_0$. Suppose that $\kappa$ and $\delta$ are real numbers satisfying $\delta>1$ and
\begin{equation}\label{eqn.kappa delta reln}
\kappa>(1+\delta)(12/5+\epsilon_2),
\end{equation}
and set $x=n^\kappa$ and $T=n^\delta-2$. Then we have that
\begin{align*}
&\int_0^{1-\eta}x^{\alpha-1}N_n(\alpha,T)\log x~d\alpha\\
&\qquad\le\kappa\int_0^{1-\eta}n^{(1-\alpha)((1+\delta)(12/5+\epsilon_2)-\kappa)}\log n~d\alpha\\
&\qquad\le\frac{\kappa}{\kappa-(1+\delta)(12/5+\epsilon_2)}\cdot n^{\eta((1+\delta)(12/5+\epsilon_2)-\kappa)}
\end{align*}
For the other term appearing in (\ref{eqn.non sieg zero int 1}) we have that
\[\frac{N_n(0,T)}{x}\le n^{(1+\delta)(1+\epsilon_2)-\kappa},\]
and this tends to $0$ as $n\rar\infty$.
For the secondary error term appearing in (\ref{eqn.psi diff err term 1}) we have that
\[\frac{x(\log nx)^2}{T}+x^{1/4}\log x\ll n^{\kappa-\delta}(\log n)^2+n^{\kappa/4},\]
and this will be $o(h\beta_0^2/\varphi (n))$ provided that
\begin{equation}\label{eqn.h bound}
n^{1+\kappa-\delta}(\log n)^2+n^{1+\kappa/4}=o(h\beta_0^2)~\text{ as }~n\rar\infty.
\end{equation}
Assuming that $h$ is chosen so that (\ref{eqn.h bound}) is satisfied, we have that
\begin{align*}
&\psi(x+h,a,n)-\psi (x,a,n)\\
&\quad\ge \frac{h}{\varphi (n)}\left(1-\frac{\kappa}{\kappa-(1+\delta)(12/5+\epsilon_2)}\cdot n^{\eta((1+\delta)(12/5+\epsilon_2)-\kappa)}\right.\\
&\quad\qquad\qquad\qquad\left.\phantom{\frac{1}{2}}-n^{-\kappa\beta_0/\log n}+o(\beta_0^2)\right).
\end{align*}
Now what we are going to show is that, for choices of $x$ and $h$ corresponding to each of our theorems, we can choose $\delta$ so that (\ref{eqn.kappa delta reln}) and (\ref{eqn.h bound}) are satisfied and such that
\begin{align}
1-\frac{\kappa}{\kappa-(1+\delta)(12/5+\epsilon_2)}\cdot n^{\eta((1+\delta)(12/5+\epsilon_2)-\kappa)}-n^{-\kappa\beta_0/\log n}\sim\kappa\beta_0,\label{eqn.sieg zero greater 0 bnd}
\end{align}
as $n\rar\infty$.

To prove Theorem \ref{thm.primes1} (in the case where there is a Siegel zero) take $h=\epsilon n^\kappa$. In this case we suppose that $\kappa>24/5$, let $\delta$ be any number satisfying
\begin{equation}\label{eqn.kappa delta reln 2}
\kappa>(1+\delta)\cdot 12/5,
\end{equation}
and then choose $\epsilon_1$ in (\ref{eqn.sieg thm}) so that (\ref{eqn.h bound}) holds. We will show that $\delta$ can be chosen so that (\ref{eqn.sieg zero greater 0 bnd}) also holds. By taking $n$ large enough, $\beta_0$ can be assumed to be as close to $0$ as necessary, and we therefore have that
\[1-n^{-\kappa\beta_0/\log n}\sim\kappa\beta_0\quad\text{as}\quad n\rar\infty.\]
Now we have that
\[n^{\eta((1+\delta)(12/5+\epsilon_2)-\kappa)}=\beta_0^{c_2(\kappa/(1+\delta)-(12/5+\epsilon_2))},\]
and this will be $o(\beta_0)$ as $n\rar\infty$, provided that
\begin{equation}\label{eqn.kappa delta reln 3}
\kappa>(1+\delta)(12/5+\epsilon_2+c_2^{-1}).
\end{equation}
Since $\delta$ can be taken as close to $1$ as necessary and $\epsilon_2>0$ can be taken as small as necessary, we will be able to choose these parameters so that (\ref{eqn.kappa delta reln}) holds for all sufficiently large $n$, provided that $\kappa>39/5$. In this case what we have shown is that
\[\psi(x+h,a,n)-\psi (x,a,n)\gg\frac{h\beta_0}{\varphi (n)}\gg n^{\kappa-1-\epsilon_1}.\]
The contribution from higher powers of primes is bounded by
\[\sum_{p\le x+h}\sum_{\substack{m\ge 2\\p^m\le x+h}}\log p\ll (x+h)^{1/2}\log x\ll n^{k/2}\log n,\]
and this verifies the first part of the statement of Theorem \ref{thm.primes1}, for all $\kappa>39/5$, in the case when there is a Siegel zero (the constant $13.4$ comes from the non-Siegel zero case).

For Theorem \ref{thm.primes2} we take $h=n^{\kappa-1-\epsilon}$. In this case we assume that $\kappa>(3+\epsilon)\cdot 12/5$, let $\delta$ be any number satisfying (\ref{eqn.kappa delta reln 2}), and then choose $\epsilon_1$ in (\ref{eqn.sieg thm}) so that (\ref{eqn.h bound}) holds. The analysis is the same as before except that in equation (\ref{eqn.kappa delta reln 3}) we are allowed to take any value of $\delta>2+\epsilon$, and this verifies the first part of Theorem \ref{thm.primes2}, for all $\kappa>(3+\epsilon)\cdot 39/10$, in the case when there is a Siegel zero.

\noindent {\it Case 2:} If there are no Siegel zeros to modulus $n$ then much of the proof is the same as before, except that we must choose
\begin{equation*}
\eta =\frac{c_3}{\log (n(2+T))},
\end{equation*}
where the constant $c_3$ may be taken to be $0.10367$ (see \cite[p.267]{Heat1992}). Proceeding as before, we have that
\begin{align*}
&\psi(x+h,a,n)-\psi (x,a,n)\\
&\quad\ge \frac{h}{\varphi (n)}\left(1-\frac{\kappa}{\kappa-(1+\delta)(12/5+\epsilon_2)}\cdot n^{\eta((1+\delta)(12/5+\epsilon_2)-\kappa)}+o(1)\right),
\end{align*}
whenever (\ref{eqn.kappa delta reln}) is satisfied and
\begin{equation*}
n^{1+\kappa-\delta}(\log n)^2+n^{1+\kappa/4}=o(h)~\text{ as }~n\rar\infty.
\end{equation*}
Now we show that, for choices of $x$ and $h$ corresponding to each of our theorems, we can choose $\delta$ so that the above conditions are satisfied and such that
\[\frac{\kappa}{\kappa-(1+\delta)(12/5+\epsilon_2)}\cdot n^{\eta((1+\delta)(12/5+\epsilon_2)-\kappa)}<1.\]

As before, for the proof of Theorem \ref{thm.primes1} we have the freedom to choose $\delta>1$ to be as close to $1$ as we like, and we have the freedom to choose $\epsilon_2>0$ to be as small as we like. Therefore we will obtain the desired result as long as
\[\frac{\kappa}{\kappa-24/5}\cdot n^{\eta(24/5-\kappa)}=\frac{\kappa}{\kappa-24/5}\cdot\exp (c_3(12/5-\kappa/2))<1.\]
The function on the left hand side decreases as $\kappa$ increases, and the inequality is satisfied for all $\kappa>13.4$.

For the proof of Theorem \ref{thm.primes2} we can choose $\delta>2+\epsilon$ and we will be able to obtain the required result provided that
\[\frac{\kappa}{\kappa-(3+\epsilon)\cdot 12/5}\cdot\exp (c_3(12/5-\kappa/(3+\epsilon)))<1.\]
When $\epsilon=0$ this inequality will be satisfied for all $\kappa>20.1$. Therefore for any $\kappa>20.1$, it is possible to choose $\epsilon>0$ small enough so that the inequality is satisfied.

\section{Proofs of Theorems \ref{thm.perpnts3.1} and \ref{thm.perpnts3.2}}\label{sec.conclusion}
In this final section we will explore the limits of the $\F$-method in constructing automorphisms with various logarithmic growth rates of periodic points. We begin by demonstrating why certain growth rates are not possible. Assume that we naively apply the $\F$-method with all of the integers $p_n$ greater than $1$. Then, since each $p_n=1\mo n$,
\[\log F_n(T)=\sum_{d|n}\log p_d>\sum_{d|n}\log d=\frac{d(n)\log n}{2}.\]
It is well known (see \cite[Theorem 317]{HardWrig2008}) that
\[d(n)>\exp\left(\frac{\log n}{2\log\log n}\right)~\text{for infinitely many}~n\in\N,\]
and this establishes a lower bound (albeit on a very thin subsequence) for the growth of $F_n(T)$. Note that as $n\rar\infty$ the function
\[\exp\left(\frac{\log n}{2\log\log n}\right)\]
grows more quickly than any power of $\log n$.

This example is taking advantage of the fact that there is a thin set of integers which have an unusually large number of divisors. We may hope to taper the growth of $F_n(T)$ by cleverly employing the $\F$-method with some of the $p_n$ equal to $1$. However it is not difficult to prove that even in this case there are arbitrarily large growth rates which are not attainable.
\begin{proof}[Proof of Theorem \ref{thm.perpnts3.1}]
Using notation as in the statement of the Theorem, and assuming without loss of generality that $t(n)\rar\infty$ as $n\rar\infty$, define $r:\N\rar\R$ by
\begin{align*}
r(n)=\begin{cases} t(n^2)^3&\text{if $n$ is prime},\\ t(n)^2&\text{else.}\end{cases}
\end{align*}
If, for some $T$ constructed using the $\F$-method, $\log F_n(T)$ were asymptotic to $r(n)$ then we would have, for all large enough integers $n$, that
\[r(n)/2 < \log F_n(T) < 2r(n).\]
Suppose this is the case, let $q_1$ and $q_2$ be two large enough primes, and let $n=q_1q_2$. Then, using formula (\ref{eqn.F_n formula}),
\begin{align*}
2t(n)^2&> \log F_n(T)= \log F_{q_1}(T) + \log F_{q_2}(T) + \log p_n\\
&> (1/2)\left(t(q_1^2)^3+t(q_2^2)^3\right)> t(n)^3,
\end{align*}
which is impossible as $n\rar\infty.$
\end{proof}
On the extreme of slowly growing numbers of periodic points, it is certainly not possible using the $\F$-method to obtain polynomial (or sub-polynomial) growth. To see why, suppose that
\[\lim_{n\rar\infty}\frac{\log F_n(T)}{k\log n}=1\]
for some $k\in\N$. Then for all large enough $n$ there is a number $\xi_n\in (-1,1)$ such that
\[\log F_n(T)=(1+\xi_n(8k)^{-1})k\log n.\]
If $q$ is a prime which is large enough and if $a\in\N$ then we have
\begin{align*}
\log p_{q^{a+1}}&=\log F_{q^{a+1}}(T)-\log F_{q^a}(T)\\
&=k\log q+\frac{((a+1)\xi_{q^{a+1}}+a\xi_{q^a})}{8}~\log q.
\end{align*}
Setting $a=2k-1$, we have the bound
\[\left|\frac{((a+1)\xi_{q^{a+1}}+a\xi_{q^a})}{8}~\log q\right|<(k/2)\log q,\]
which implies that
\[q^{k/2}<p_{q^{2k}}<q^{3k/2}.\]
However this is impossible, since we require that $p_{q^{2k}}=1\mo q^{2k}$.

In contrast to the above argument, we conclude with our proof of Theorem \ref{thm.perpnts3.2}, in which we explain how to construct automorphisms with growth rates of periodic points only marginally asymptotically faster than polynomial.
\begin{proof}[Proof of Theorem \ref{thm.perpnts3.2}]
With notation as in the statement of the Theorem, let
\[r(n)=\iota (n)\log n.\]
Taking $\kappa=13.5$ in Theorem \ref{thm.primes1}, it follows that there is a constant $c>1$ with the property that, for any $n\in\N$ and for any $x>cn^\kappa$, there is a prime number in the interval $[x,2x)$ which is $1$ modulo $n$. We apply the $\F$-method to choose the integers $p_n$ as follows.

If $n$ is not a prime power then we set $p_n=1$. For each prime $q$ and for each $a\in\N$ we choose $p_{q^a}$ so that
\begin{equation}\label{eqn.p_q^a ineq}
\left|\sum_{i=1}^a\log p_{q^i}-r(q^a)\right|\le\log (cq^{a\kappa}).
\end{equation}
To verify that this is possible we argue for each prime $q$ by induction on $a$. First of all if
\[r(q)\le\log (cq^{\kappa}),\]
then we take $p_q=1$, otherwise we choose $p_q$ to be a prime equal to $1\mo q$ which lies in the interval
\[\left[\exp(r(q)),2\exp(r(q))\right).\]
In the second case we have that
\[\left|\log p_q-r(q)\right|\le \log 2,\]
so (\ref{eqn.p_q^a ineq}) is satisfied with $a=1$. Now suppose the inequality holds for some $a\in\N$. Then if
\begin{align*}
\left|\sum_{i=1}^a\log p_{q^i}-r(q^{a+1})\right|\le \log (cq^{(a+1)\kappa}),
\end{align*}
we take $p_{q^{a+1}}=1$. Otherwise (since $r$ is non-decreasing and the inequality was assumed to be true for $a$) it must the case that
\[\delta_{a+1}:=r(q^{a+1})-\sum_{i=1}^a\log p_{q^i}>\log (cq^{(a+1)\kappa}),\]
and we may choose $p_{q^{a+1}}$ to be a prime equal to $1\mo q$ in the interval $[\delta_{a+1},2\delta_{a+1})$. Then as before we have that
\[\left|\sum_{i=1}^{a+1}\log p_{q^i}-r(q^{a+1})\right|\le |\delta_{a+1}-\log p_{q^{a+1}}|<\log 2,\]
which verifies (\ref{eqn.p_q^a ineq}).

Finally, suppose that $n\in\N$ and write its prime factorization as $n=q_1^{a_1}\cdots q_k^{a_k}$. Then we have that
\begin{align*}
\log F_n&=\sum_{d|n}\log p_d\\
&=\sum_{i=1}^k\sum_{a=1}^{a_i}\log p_{q_i^a}\\
&=\sum_{i=1}^k\left(r(q_i^{a_i})+O(\log (q_i^{a_i}))\right)\\
&=\sum_{i=1}^k r(q_i^{a_i})+O(\log n).
\end{align*}
Letting $\mathcal{L}(n):=\log\log n$, we have that
\begin{align}
\left|\iota (n)\log n-\sum_{i=1}^k r(q_i^{a_i})\right|&=\sum_{i=1}^k (\iota (n)-\iota (q_i^{a_i}))\log (q_i^{a_i})\nonumber\\
&\le\sum_{\substack{i=1\\q_i^{a_i}>\mc{L}(n)}}^k2\log (q_i^{a_i})+\sum_{\substack{i=1\\q_i^{a_i}\le\mc{L}(n)}}^k\iota (n)\log (q_i^{a_i}),\label{eqn.pf of thm perpnts3.2}
\end{align}
where we have used the inequalities $|\iota (n)-\iota (q_i^{a_i})|<2$ for $q_i^{a_i}>\mc{L}(n)$, and $|\iota (n)-\iota (q_i^{a_i})|<\iota (n)$ otherwise. The first sum on the right hand side is bounded above by $2\log n$. For the second we use the Prime Number Theorem (or even Chebyshev's Theorem) to obtain the bound
\[\sum_{\substack{i=1\\q_i^{a_i}\le\mc{L}(n)}}^k\iota (n)\log (q_i^{a_i})\ll\iota (n)\log\log n.\]
Therefore the right hand side of (\ref{eqn.pf of thm perpnts3.2}) is $o(\iota (n)\log n)$, and the proof is completed.
\end{proof}

\end{document}